\PassOptionsToPackage{usenames, dvipsnames}{color}
\documentclass[reqno]{amsart}
\usepackage{mathrsfs}
\usepackage{amsmath}
\usepackage{bbm}
\usepackage{amsthm}
\usepackage{amsfonts}
\usepackage{amssymb}
\usepackage{url}
\usepackage{tikz}
\usepackage{pgfplots} 
\usepackage{enumerate}
\usepackage[pdftex,bookmarks=true]{hyperref}
\usepackage[usenames, dvipsnames]{color}
\usepackage{verbatim}
\usepackage{amsmath}

\usepackage{pdfsync}

\newcommand\rank{{\operatorname{rank}}}
\newcommand\corank{{\operatorname{corank}}}

\newcommand\R{{\mathbb{R}}}
\newcommand\C{{\mathbb{C}}}

\renewcommand\P{{\mathbb{P}}}
\newcommand\E{{\mathbb{E}}}

\newcommand\Z{{\mathbb{Z}}}

\newcommand\F{{\mathbb{F}}}

\newcommand\Bv{{\mathbf v}}
\newcommand\Bw{{\mathbf w}}

\newcommand\CE{{\mathcal E}}

\newcommand\supp{\mathbf{supp}}

\newcommand\eps{\varepsilon}

\renewcommand\a{\alpha}

\newcommand\Aut{\mathbf{Aut}}

\usepackage{fullpage}

\theoremstyle{plain}
  \newtheorem{theorem}[subsection]{Theorem}

  \newtheorem{proposition}[subsection]{Proposition}
  \newtheorem{fact}[subsection]{Fact}
  \newtheorem{lemma}[subsection]{Lemma}
  \newtheorem{corollary}[subsection]{Corollary}

\theoremstyle{definition}
  \newtheorem{definition}[subsection]{Definition}

\makeatletter
\renewcommand*\env@matrix[1][*\c@MaxMatrixCols c]{%
  \hskip -\arraycolsep
  \let\@ifnextchar\new@ifnextchar
  \array{#1}}
\makeatother

\begin{document}

\title{The Rank of the Sandpile Group of Random Directed Bipartite Graphs}
\author{Atal Bhargava, Jack Depascale, Jake Koenig}


\begin{abstract} 
We identify the asymptotic distribution of $p$-rank of the sandpile group of random directed bipartite graphs which are not too imbalanced. We show this matches exactly that of the Erd{\"o}s-R{\'e}nyi random directed graph model, suggesting the Sylow $p$-subgroups of this model may also be Cohen-Lenstra distributed. Our work builds on results of Koplewitz who studied $p$-rank distributions for unbalanced random bipartite graphs, and showed that for sufficiently unbalanced graphs, the distribution of $p$-rank differs from the Cohen-Lenstra distribution. Koplewitz \cite{K} conjectured that for random balanced bipartite graphs, the expected value of $p$-rank is $O(1)$ for any $p$. This work proves his conjecture and gives the exact distribution for the subclass of directed graphs.
\end{abstract}
\maketitle
\section{Introduction} 
Given a directed  graph $G$, we can associate with it an Abelian group $\Gamma(G)$ defined as 
\[
\Gamma(G) = \Z_0^n/M\Z^n
\]
where $M$ is the Laplacian matrix of $G$ and $\Z_0^n$ is the subspace of $\Z^n$ orthogonal to the all $1$'s vector i.e. those vectors with zero sum. This group is known as the sandpile group, as well as the the Jacobian, or critical group of $G$.

A well studied question about the sandpile group is its asymptotic structure for random graphs. The sandpile group of the Erd\H{o}s-R\'enyi random graph is well understood. In \cite{W0}, Wood showed for $G \in G(n,q)$ an Erd\H{o}s-R\'enyi random graph, the $p$-parts of $\Gamma(G)$ are independent for a finite set of $p$ and are isomorphic to a given $p$-group $H$ with probability proportional to,
\begin{align}\label{symCL}
    \frac{\#\{\text{symmetric, bilinear, perfect } \phi:H\times H\to \C^*\}}{|\Aut(H)|}.
\end{align}
This distribution is related to the Cohen-Lenstra distribution which samples a given $p$-group with probability inversely proportional to the size of its automorphism group. Explicitly a random $p$-group $G$ is Cohen-Lenstra distributed if it's probability of being equal to a fixed $p$-group $H$ is given by,
\begin{align*}
    \P(G\simeq H) = \frac{1}{|\Aut H|}\prod_{i=1}^\infty (1-1/p^{i}).
\end{align*}
The distribution was first observed by Cohen and Lenstra as the empirical distribution of the class group of real quadratic number fields. It also appears as the distribution of the cokernel of random matrices of independent entries. This was shown first for Haar measure by Friedman and Washington \cite{FW} and later for a more general class of distributions by Maples \cite{M2} and Wood \cite{W1}. We will refer to the distribution given by Equation~\ref{symCL} as the symmetric Cohen-Lenstra Distribution.

Another important class of random graphs is those with vertices of a specified fixed degree.  M\'{e}sz\'{a}ros \cite{Mes} showed for $d\geq 3$ random $d$-regular directed graphs have sandpile groups with Cohen-Lenstra distribution. Undirected $d$-regular graph's sandpile groups have the same distribution as in Equation~\ref{symCL} for $d$ odd. The distribution for even $d$ is different but related. See M\'{e}sz\'{a}ros' paper for details.

Though ubiquitous, not all random groups are Cohen-Lenstra distributed. In \cite{K}, Koplewitz shows that the distribution of $p$-rank of sandpile groups of random bipartite graphs on vertex sets $V_1,V_2$ depends on the ratio of the vertex sets $\a = \frac{|V_1|}{|V_2|}\leq 1$. He showed that for $\a < 1/p$ the expected $p$-rank is proportional to $n$ and at the critical threshold $\a = 1/p$ the $p$-rank is proportional to $\sqrt{n}$. So instead of the Cohen-Lenstra distribution one doesn't get a limiting distribution at all.

For more balanced bipartite graphs where $1>\a>1/p$, Koplewitz showed the expected $p$-rank is $O(1)$ where the constant depends on $p$. Koplewitz's method doesn't extend to the balanced case $\a=1$ but one can infer from the result that the expected $p$-rank is $o(n)$ in that case. Koplewitz conjectured in the balanced case the expected $p$-rank is $O(1)$. While the work of Koplewitz applies to the more challenging case of undirected graphs, his methods generalize easily to the directed case as well. In this work, we are able to achieve Koplewitz's conjectured bound and find the exact limiting distribution of $p$-rank in the case of directed bipartite graphs with $\alpha>1/p$.

The result of this paper improves Koplewitz's result in two ways: it applies to the balanced case ($\alpha = 1$) and it gives the limiting $p$-rank distribution, not only its expectation.

Our method and result are similar in spirit to other studies of matrices with fully iid entries over finite fields (\cite{M1, M2,NgP,W1}). For instance, it is known (see \cite{M2,W1}) that under mild min-entropy conditions, for a random $n\times (n+u)$ matrix $A$ with iid entries over $\F_p$, we have
\[
\P(\corank(A)=k) =  \frac{1}{p^{k(u+k)}} \frac{\prod_{i=k+1}^\infty (1-1/p^i)}{ \prod_{i=1}^{k+u}(1-1/p^i)} + \exp(-Cn)
\]
We achieve essentially the same result for the matrix model of the Laplacian of a not too imbalanced random directed bipartite graph, which we describe below.

Let $1\geq \alpha>1/p$. Consider a random directed bipartite graph on sets of $n$ and $\lfloor\alpha n\rfloor$ vertices (in the future we will omit the floors when they make no essential difference) given by including each crossing edge independently with probability $0<q<1$, a constant independent of $n$. Denote the Laplacian of this graph by $M$. Note the rows of this matrix are independent (in the probabilistic sense).

\begin{theorem}\label{theorem:main} Let $p$ be a prime such that $p\ll \sqrt{n}$. Then for any $k\le (1+\alpha)n$ we have
$$\P(\corank(M/p)=1+k) =  \frac{1}{p^{k^2+k}} \frac{\prod_{i=k+2}^\infty (1-1/p^i)}{ \prod_{i=1}^{k}(1-1/p^i)} + \exp(-Cn/p^2).$$
\end{theorem}

Based on computations and intuition we conjecture the rank of the sandpile group for undirected bipartite graphs should have the distribution predicted by Equation~\ref{symCL} for $p>2$. However our method depends on the independence between the rows. We further conjecture that the cokernels in the directed and undirected model are actually Cohen-Lenstra and symmetric Cohen-Lenstra distributed.

\subsection{Notation} 

For $J$ a set of indices we denote by $M^J$ the restriction of the matrix to columns in $J$ and $M_J$ the restriction of the matrix to rows in $J$. When $J = [k]$ we abuse the notation by writing $M^k$ or $M_k$. In particular $M_n$ is the top half of our matrix and $M_{(1+\alpha)n}$ is $M$.

Generally $W$ will denote the row space of $M$ with the same subscripting convention and $\beta = \min(q, 1-q)$. But we'll define them in place before using them.

We use $\mathbbm{1}$ to denote the all $1$'s vector.

We fix a prime $p$ and work mod $p$ for all of our results. We treat $p$ as a constant but our arguments work for $p=o(\sqrt{n})$. We note our results don't necessarily hold when $p$ grows faster with $n$. We use $C$ to denote constants, not necessarily the same, which may depend on $p,q$ and $\alpha$.
\subsection{Structure of Paper} Our method is a row exposure process similar to the method by Maples \cite{M2} and Nguyen and Paquette \cite{NgP}. Note that we use a row exposure process instead of a column exposure process because as the Laplacian of a directed graph is traditionally defined the rows are independent but not the columns.

In Section~\ref{fullrk} we show that $M_{(1+\alpha)n-r}$ is full rank with high probability with respect to $r$. We make heavy use of a variant of Odlyzko's lemma, Lemma~\ref{lemma:odlyzko}.

In Section~\ref{rankevolution} we show that the remaining $r$ rows increase the rank of the matrix with probability similar to what one would predict with the uniform model. The argument will make use of a notion of structure and results about it due to Meehan, Luh and Nguyen \cite{LMNg}. The technique is modified for the Laplacian setting.

In Section~\ref{finalproof} we combine these two facts to show the distribution of the rank of $M$ is exponentially close to the distribution of the rank of a uniform random matrix.
\subsection{Acknowledgement}

Thanks to Hoi Nguyen for leading us in a summer reading group which lead us to this problem. Thanks to Nathan Kaplan for pointing out some important omissions, errors and ambiguities in an earlier draft. Finally, we thank the anonymous reviewer for corrections and insightful comments.

\section{Full rank until final rows}\label{fullrk} 

In this section we prove that $M_{(1+\alpha)n-r}$ is full rank with probability at least $1-\exp(-C r)$. 

We reintroduce a concept from Koplewitz \cite{K} itself a variant of a definition from Maples \cite{M2} called min-entropy. It combines a nonconcentration property with a weak notion of independence. 

\begin{definition}
Let $A$ be a random matrix over $\Z/p\Z$, $\beta>0$ and $I$ a set of entries in $A$. We say that $a_{ij}$ has min-entropy $\beta$ with respect to $I$ if for any conditioning on a possible set of values for the entries in $I$ the probability that $a_{ij}$ takes on any value in $\Z/p\Z$ is bounded above by $1-\beta$.

We say the matrix or vector has min-entropy $\beta$ if every entry of the matrix or vector has min-entropy $\beta$ with respect to the set of other entries. We say a subset of entries has min-entropy $\beta$ if every entry in the subset has min-entropy $\beta$ with respect to the other entries in the subset.
\end{definition}

The Laplacian has min-entropy $0$ because the entries in a row off the diagonal determine the value of the diagonal. Instead we have the following.
\begin{fact}\label{minentropyforlaplacian}
The diagonal entries of $M_n$ together with the entries of $M_n$ in columns $n+1$ to $(1+\alpha)n-p$ have min-entropy $\min(q,1-q)^p$.
\end{fact}
Many of the lemmas in this section make use of the following variant of a lemma of Odlyzko \cite{O0} and its corollary. Corollary~\ref{fullrankforrects} is Theorem~10 in \cite{K} and Lemma~\ref{lemma:odlyzko} is a statement in its proof.
\begin{lemma}[Odlyzko]\label{lemma:odlyzko}
Let $V$ be a subspace of $\F_p^n$ of codimension $k$. Then, for a random vector $X = (x_i)_{i=1}^n\in \F_p^n$ with min-entropy $\beta$, we have
\[
\P(X\in V) \leq (1-\beta)^k.
\]
\end{lemma}

\begin{corollary}\label{fullrankforrects}
An $n\times m$, $n>m$ matrix of min-entropy $\beta$ is full rank with probability at least $$1-\frac{1}{\beta^2}(1-\beta)^{n-m}.$$
\end{corollary}

These facts are not directly applicable in the Laplacian model, as these models inherently yield matrices/vectors with 0 min-entropy. But they will be useful applied to trimmed submatrices/subvectors which have nontrivial min-entropy.

\begin{lemma}\label{lemma:first-n-rows} There exists $C>0$ such that with probability at least $1-O(\exp(-Cn))$, the matrix $M_n$ has full rank.
\end{lemma}
\begin{proof}
Write,

\[
\renewcommand\arraystretch{1.3}
M_n = \left(\begin{array}{cccc|cccc}
-\delta_1 & 0 & \hdots & 0 &  x_{1,1} & x_{1,2}& \hdots & x_{1,\alpha n}\\
0 & -\delta_2& \hdots & 0 & x_{2,1} & x_{2,2} & \hdots & x_{2,\alpha n}\\
\vdots & \vdots & \ddots & \vdots &   \vdots & \vdots & \ddots & \vdots\\
0 & 0 &\hdots & -\delta_n &  x_{n,1}  & x_{n,2} & \hdots & x_{n,\alpha n}\\
\end{array}\right)
\]

where $\delta_j =\sum_{i=1}^{\alpha n}x_{j,i}$. Notice that the set of rows with $\delta_j\neq 0$ are linearly independent and a row with $\delta_j=0$ is not in its span (unless the row is \textbf{0}). So it suffices to show that with high probability, the rows with $\delta_j=0$ are linearly independent.

Towards that, we show there aren't many rows with $\delta_i=0$. Let $I\subseteq[n]$ be the set of all $j$ with $\delta_j = 0$. Lemma~\ref{lemma:zerosum} implies
\begin{align*}
    \P(i\in I)  \leq 1/p + \eps'.
\end{align*}
Where $\eps' = \exp(-\alpha n/2p^2)$. Noting that the rows are probabilistically independent and applying a Chernoff bound, we have for any $\eps>0$,
\[
\mathop{\mathbb{P}}(|I|\leq (1/p+ \eps' + \eps) n)\geq 1-\exp(-\eps^2n).
\]

Next we show that $M_I$, the submatrix generated by the rows in $I$, is full rank with high probability. To do this, we consider only the columns in $J = \{n+1,\dots,(1+\alpha)n\}$, and show $M_I^J$ is also full rank with high probability. Observe that $M_I^J$ is $|I|\times \alpha n$, but the sum of each row is already conditioned to be 0. So, truncate the final $p$ columns to obtain an $|I|\times (\alpha n-p)$ matrix which by Fact~\ref{minentropyforlaplacian} has min-entropy $\beta = \min(q,1-q)^p$. Therefore by Lemma~\ref{fullrankforrects} this matrix is full rank with probability $1-\frac{1}{\beta^2}(1-\beta)^{-(\alpha n-p-|I|))}$.

Conditioning on $|I|\leq (1/p + \eps' + \eps) n$, the rows in such a matrix are linearly independent with probability at least $1 - O(\exp(-C(\alpha-1/p-\eps' - \eps)n))$. 
Combining this with the above, taking $\eps = \alpha - 1/p - \eps'$, which is positive for sufficiently large $n$,
\[
\P(\rank(M_n) = n)\geq 1-O(\exp(-C(\alpha - 1/p)n)) - \exp(-(\alpha-1/p - \eps')^2n).
\]
Note it is crucial that $\alpha>1/p$ for this result. In the case $\alpha<1/p$, $M_I^J$ is not full rank with high probability and the lemma is false.

\end{proof}
\begin{lemma}\label{lemma:zerosum}
For any $n$, $p\ll \sqrt{n}$, $\left|\P\left(\sum_{i=1}^nx_i = 0\right)-\frac{1}{p}\right| \leq \exp(-n/2p^2)$.
\end{lemma}
\begin{proof}
Note,
\begin{align*}
    \left|\P\left(\sum_{i=1}^n x_i = 0\right)-\frac{1}{p}\right| \leq \sup_{a\in\F_p} \left|\P\left(\sum_{i=1}^n x_i = a\right)-\frac{1}{p}\right|.
\end{align*}
This is equal to $\rho(\mathbbm{1})$ as defined in \cite{LMNg}. Therefore the non-Laplacian version of Lemma \ref{laplacian_concentration}, which is Corollary 4.6 from that paper, gives the result.
\end{proof}
 
\begin{proposition}\label{prop:rect-full-rank}
There exists $C>0$ such that for any $r<n$ the matrix $M_{(1+\alpha)n-r}$ generated by rows $X_1,\dots,X_n,Y_{1},\dots,Y_{\alpha n-r}$ is full rank with probability $1-O(\exp(-Cr))$.
\end{proposition}
\begin{proof}
Letting $J=\{n+1,\dots,  n+\alpha n\}$, we see $M_n^J$ has iid entries and thus by Corollary \ref{fullrankforrects} has rank $\alpha n-r'$ with probability at least $1-\exp(-Cr')$ for any $r'<\alpha n$. Condition on $M_n^J$ having rank $\alpha n - r'$, $r'$ to be chosen as a function of $r$.

We proceed inductively. By Lemma~\ref{lemma:first-n-rows}, $M_n$ is full rank with probability at least $1-\exp(-Cn)$. Condition on $\rank(M_{n+i-1}) = n+i-1$ for some $1\leq i\leq \alpha n - r$. We will show $\rank(M_{n+i}) = n+i$ with probability $1-O(\exp(-Cr))$.

Let $I = J\setminus {n+i}$ and let $P_{I}$ be the projection onto the coordinates in $I$. As $Y_i$ vanishes on $I$, $Y_i\in W_{n+i-1}$ if and only if it is in the kernel of the $P_I$ restricted to $W_{n+i-1}$ e.g. $Y_i\in \ker(P_I|_{W_{n+i-1}})$. Where $W_{n+i-1}$ is the row space of $M_{n+i-1}$. Using our assumption $M_n^J$ has rank at least $\alpha n - r'$, $M_n^I$ has rank at least $\alpha n - r' - 1$. Using our inductive hypothesis $\dim(W_{n+i-1}) = n+i-1$ and the rank-nullity theorem we have,

\begin{align*}
    \dim P_{[n]}(\ker(P_I|_{W_{n+i-1}})) &\leq \dim\ker(P_I|_{W_{n+i-1}}) \\ &= \dim W_{n+i-1} - \dim P_I W_{n+i-1} \\&\leq n+i-1 - (\alpha n - r' - 1)
\end{align*}

Therefore the codimension of $P_{[n]}(\ker P_I|_{W_{n+i-1}})$ as a subspace of $\F_p^n$ is at least,
\begin{align*}
    n - (n+i-1 - (\alpha n - r' - 1)) = \alpha n - i - r'.
\end{align*}

Because $Y_i$ has iid entries in the first $n$ indices, $Y^n_i$ has min-entropy $\beta = \min(q, 1-q)$. Applying Lemma \ref{lemma:odlyzko},

\begin{align*}
\P(Y_{i}\in W_{n+i-1})\leq \P(Y_{i}^{n}\in P_{[n]}(\ker P_I|_{W_{n+i-1}}))\leq \beta^{\alpha n-r'-i}.
\end{align*}
Summing the loss in probability $\dim(W_{n+i}) = n+i$ at each step,
\[
\P(M_{(1+\alpha)n-r}\text{ is full rank})\geq 1-\sum_{i=1}^{\alpha n-r}\beta^{\alpha n-r'-i}.
\]
Taking $r' = r/2$ we achieve the desired result where
\[
1-\sum_{i=1}^{\alpha n-r}\beta^{\alpha n-\frac{r}{2}-i} = 1-O(\exp(-Cr)).
\]
\end{proof}

\section{Rank evolution}\label{rankevolution} 

Our goal for this section is to show the rank evolution as the final $r$ rows are exposed is what one would predict from the uniform model.

Let $W_{(1+\alpha)n-k}$ be the subspace generated by the first $(1+\alpha)n-k$ rows of the matrix $M_{(1+\alpha)n}$. 

\begin{proposition}[Rank evolution]\label{prop:rank} Let $\delta>0$, and $k\leq \delta n$. There exists an event $\CE_{(1+\alpha)n-k}$ of probability at least $1 -O(\exp(-Cn))$ such that the following holds.
$$\P(Y_{(1+\alpha)n-k+1} \in W_{(1+\alpha)n-k} | \CE_{(1+\alpha)n-k} \wedge W_{(1+\alpha)n-k} \mbox{ has codimension $l$}) = \frac{1}{p^{l-1}} + O(\exp(-Cn/p^2)).$$
\end{proposition}

This proposition is established by showing the normal vectors to the row space of the first $(1+\alpha-\delta)n$ rows are not sparse or structured with high probability. That is Proposition~\ref{prop:non-sparse} and Proposition~\ref{laplacian_concentration} respectively. Along with a fact about subspaces with unstructured normal vectors, Lemma~\ref{smallrhogood}, we can bound the transition probabilities.

In the following proposition we show the row space of $M_{(1+\alpha - \delta)n}$ does not have sparse normal vectors with high probability. Even stronger we show their projection to the first $n$ coordinates is not sparse. This is necessary as the last couple of rows only have min-entropy in the first $n$ coordinates.

\begin{proposition}[Non-sparseness of normal vectors]\label{prop:non-sparse} Let $\alpha>\delta>0$ be given. There exists $\delta'>0$, $C>0$ such that with probability at least $1-O(\exp(-Cn))$, any vector $\Bv\in \F_p^{(1+\alpha)n}\setminus \F_p\cdot\mathbbm{1}$ that is orthogonal to $X_1,\dots, X_n,Y_1,\dots, Y_{\alpha n-\delta n}$ and any $a\in \F_p$, we have 
$$|\supp\left(\Bv-a\cdot\mathbbm{1}\right) \cap [n]| \ge \delta' n.$$
\end{proposition}

\begin{proof}
Let $W$ be the rowspace of $M_{(1+\alpha - \delta)n}$. Because $M$ is a Laplacian, for any $a\in \F_p$, we have that $a\cdot\mathbbm{1}\in W^\perp$. Thus, $W^\perp-a\cdot\mathbbm{1} = W^\perp$ which allows us to reduce to the case of $a=0$.

Let $J = \{n+1,\dots, \alpha n\}$ and consider an arbitrary subset $I\subseteq [(1+\alpha)n]$ with $J\subseteq I$ and $|I| = \alpha n+m$. Then, we consider the matrix
\[
M_{(1+\alpha - \delta)n}^I =
\begin{pmatrix}
    D_1 & A_1\\
A_2 & D_2\\
\end{pmatrix}.
\]
Where $A_1,$ $A_2$ are iid of dimension $n\times \alpha n$ and $(\alpha -\delta)n \times m$ respectively. Note that $D_1$, $D_2$ are not necessarily diagonal, but trimmed versions of the diagonal matrices consisting of corresponding row sums from $A_1$, $A_2$.

We prove $M_{(1+\alpha - \delta)n}^I$ is full rank with probability $1-O(\exp(-Cn))$. To do this, we argue that there is a rank $n$ subset of rows comprised primarily from the first $n$ rows (i.e., rows from $A_1,$ $D_1$).

By Lemma \ref{lemma:odlyzko}, $A_1$ has rank at least $\a n-\eps n$ with probability at least $1-O(\exp(-Cn))$ for any constant $\eps>0$. Further trimming the leftmost $m$ columns, consider
\[
M_{(1+\alpha - \delta)n}^{J} =
\begin{pmatrix}
A_1\\
D_2\\
\end{pmatrix}.
\]

By the same argument as Lemma~\ref{lemma:first-n-rows} this matrix is full rank with probability at least $1-O(\exp(-Cn))$. The only difference from Lemma $\ref{lemma:first-n-rows}$ is that the entries in $D_2$ are independent from the entries in $A_1$.

By the union bound, $M_{(1+\alpha -\delta)n}^{J}$ will have full rank and $A_1$ will have rank at least $\alpha n-\eps n$ simultaneously with probability at least $1-O(\exp(-Cn))$. We then take the linearly independent $\alpha n-\eps n$ rows from $[n]$ and supplement via some $\eps n$ rows from $\{n+1,\dots,n-\delta n\}$ to get a rank $\alpha n$ minor of $M_{(1+\alpha - \delta)n}^{J}$. Let $L\subseteq [(1+\alpha)n]$ denote the indices of the rows comprising the full rank minor, $M_{L}^{J}$.

Now we expose the leftmost $m$ columns of $M_{(1+\alpha - \delta)n}^{I}$ (this is similar to Lemma~\ref{prop:rect-full-rank}). Applying Lemma~ \ref{lemma:odlyzko} to the subvector of the last $n-\delta n$ entries of the $i^{th}$ exposed vector it is in the column span of $M_L^J$ and the previously exposed columns with probability at most $\beta^{(\alpha-\delta-\eps)n-i}$, where $\beta=\min(q, 1-q)$. Using,
\begin{align*}
|L\cap \{n+1,\dots,(1+\alpha)n\}|\leq \eps n.
\end{align*}

Exposing all $m$ columns, we see that
\[
\P(M_{(1+\alpha - \delta)n}^{I}\text{ is full rank })\geq  \prod_{i=1}^m\left(1-\beta^{(\alpha-\delta-\eps)n-i}\right).
\]
If $M_{(1+\alpha - \delta)n}^{I}$ is full rank, then any nonzero orthogonal vector to the rows of $M_{(1+\alpha - \delta)n}$ must have support outside the columns of $M_{(1+\alpha - \delta)n}^{I}$. When we look at all possible choices of $I$, if all choices lead to $M_{(1+\alpha - \delta)n}^{I}$ being full rank, then there are at least $m+1$ nonzero entries in a normal vector to $M_{(1+\alpha - \delta)n}$. So, by union bound, we have that 
\[
\P(|\supp(\Bv) \cap[n]|\geq m) \geq 1- \binom{n}{m}\left(1-\prod_{i=1}^m\left(1-\beta^{(\alpha-\delta-\eps)n-i}\right)\right).
\]
Then, setting parameters, we choose $m = \delta'n$ for some $\delta'=\delta'(q)>0$ so that we have 
\[
1-\binom{n}{\delta'n}\left(1-\prod_{i=1}^{\delta' n}\left(1-\beta^{(\alpha-\delta-\eps)n-i}\right)\right)\geq 1-\binom{n}{\delta'n}\left(\delta' n \beta^{(\alpha-\delta-\eps)n-\delta' n}\right)
\]
\[
=1-O(\exp{(-Cn)})
\]
for some constant $C > 0$, where the final equality holds as
\[
\binom{n}{\delta'n}\ll d^n
\]
for $d=d(\delta')$ where we can choose $\delta'$ so that $d$ is arbitrarily close to 1 independent of $n$.
\end{proof}

We introduce the following concentration probability, designed for our purpose,
\[
\rho^j_L(\Bw) = \underset{{a\in\F_p}}{\sup}|\P[X\cdot\Bw = a]-\frac{1}{p}|,
\]
where $X = (x_1,\dots,x_n,0,\dots,-\sum_ix_i,0,\dots,0)$, with the sum at index $n+j$ and the $x_i$ are iid copies of a random variable which is 1 with probability $q$, otherwise 0. It is Laplacian version of $\rho$ from \cite{LMNg} wherein $X$ has exclusively iid entries. As the rows we add in the final $\delta n$ steps have only one nonzero entry outside the first $n$ coordinates, our notion of structure must be sensitive to structure specifically in those coordinates.

Conceptually, we think of vectors with small $\rho^j_L$ as being unstructured.

The following proposition is the Laplacian version of a structure property from \cite{LMNg}, Corollary 4.6 in their paper. The argument is similar to the one in that paper and included here for completeness.

\begin{proposition}\label{laplacian_concentration} Suppose for all $a\in\F_p$, $\Bw-a\cdot \mathbbm{1}$ has at least $m$ non-zero coordinates among the first $n$ coordinates where $p<\sqrt{m}$, then
\[
\rho^j_L(\Bw)\leq \exp{(-m/2p^2)},
\]
for all $j$.
\end{proposition}

\begin{proof}
Observe that $X\cdot \Bw = X\cdot (\Bw-a\cdot\mathbbm{1})$ for any $\Bw$ and $a\in \F_p$.

Thus, if $X = (x_1,\dots,x_n,0,\dots, -\sum_ix_i,0, \dots,0)$ where $-\sum_ix_i$ is at index $j$, it suffices to consider the case $\Bw_j = 0$ since otherwise we consider $\Bw-(\Bw_j,\dots,\Bw_j)$.

Note that we view $\Bw$ as a member of $\F_p^{(1+\alpha)n}$ with the representation 
\[
\Bw\in[-(p-1)/2,(p-1)/2]^{(1+\alpha)n}.
\]
Since $\frac{t\Bw}{p}\in\R^{(1+\alpha)n}$ has at least $m$ non-zero coordinates among the first $n$ coordinates for any $t\in\F_p$, restricting to the first $n$ coordinates, we must have that $$\Big\|\Big(\frac{t\Bw}{p}\Big)\Big|_n\Big\|_2 \geq \sqrt{m(1/p)^2} = \sqrt{m}/p.$$
Let $e_p(x) = e^{2\pi ix/p}$. Now, we proceed as in \cite{LMNg}. For any $r\in \F_p$, we have
\begin{align*}
\P(X\cdot \Bw=r)-\frac{1}{p} &= \E\left[\frac{1}{p}\sum_{t\in\F_p,t\neq 0}e_p\left(-rt + \sum_{k=1}^nx_kw_kt\right)\right] \\
&=\E\left[\frac{1}{p}\sum_{t\in\F_p,t\neq 0}e_p(-rt)\prod_{k=1}^ne_p(x_kw_kt)\right] \\
&=  \frac{1}{p}\sum_{t\in\F_p,t\neq 0}e_p(-rt)\prod_{k=1}^n\E[e_p(x_kw_kt)]. 
\end{align*}
Thus, we have
\begin{align*}
|\P(X\cdot \Bw=r)-\frac{1}{p}| &\leq \frac{1}{p}\sum_{t\in\F_p,t\neq 0}\left|\prod_{k=1}^n\E[e_p(x_kw_kt)]\right| \\
&= \frac{1}{p}\sum_{t\in\F_p,t\neq 0}\left|\prod_{k=1}^n\cos(2\pi w_kt/p)]\right|\\
&= \frac{1}{p}\sum_{t\in\F_p,t\neq 0}\left|\prod_{k=1}^n\cos(\pi w_kt/p)]\right|.
\end{align*}
Now, we observe the fact that \[
|\cos\frac{2\pi x}{p}|\leq 1-\frac{1}{2}\sin^2\frac{2\pi x}{p}\leq 1-2\|\frac{2x}{p}\|^2_{\R/\Z}.
\]
Where $\|\cdot\|_{\R/\Z}$ represents the distance to the closest integer. Then, we have that
\begin{align*}
\frac{1}{p}\sum_{t\in\F_p,t\neq 0}\left|\prod_{k=1}^n\cos(\pi w_kt/p)]\right| & \leq \frac{1}{p}\sum_{t\in\F_p,t\neq 0}\prod_{k=1}^n\left(1-2\|\frac{tw_k}{p}\|^2_{\R/\Z}\right) \\
&\leq \frac{1}{p}\sum_{t\in\F_p,t\neq 0}\prod_{k=1}^n\left(e^{-2\|\frac{tw_k}{p}\|^2_{\R/\Z}}\right) \\
&=\frac{1}{p}\sum_{t\in\F_p,t\neq 0}e^{-2\sum_{k=1}^n\|\frac{tw_k}{p}\|^2_{\R/\Z}} \\
&\leq \frac{1}{p}\sum_{t\in\F_p,t\neq 0}e^{-2\|(\frac{t\Bw}{p})|_n\|^2_2} \\
&\leq \frac{1}{p}\sum_{t\in\F_p,t\neq 0}e^{-m/2p^2} \\
&\leq e^{-m/2p^2}.
\end{align*}
\end{proof}

The following is a variant of Lemma 7.1 from \cite{LMNg} adapted for the different notion of structure in the Laplacian setting. 

\begin{lemma}\label{smallrhogood}
Let $H$ be a subspace of $\F_p^{(1+\alpha)n}$ of codimension $d$ such that $\mathbf{1}\in H^\perp$ and for any $\Bw\in H^\perp\setminus \F_q\mathbf{1}$, we have $\rho^j_L(\Bw)\leq \delta$. Then, if $X = (x_1,\dots,x_n,0,\dots,-\sum_ix_i,0,\dots,0)$ with the sum at index $n+j$ and the $x_i$ are iid variables equal to $1$ with probability $q$ otherwise $0$,
\[
|\P(X\in H)-\frac{1}{p^{d-1}}|\leq \frac{p}{p-1}\delta.
\]
\end{lemma}
\begin{proof}
Let $\Bv_1,\dots,\Bv_{d-1},\mathbf{1}$ be a basis of $H^\perp$. Then we have that for all $t_1,\dots,t_{d-1}$ not all zero and $a$,
\[
\rho^j_L(a\mathbf{1} + \sum_it_i\Bv_i)\leq \delta.
\]
Using the identity,
\[
\mathbbm{1}_{a_1=0,\dots,a_d=0} = p^{-d}\sum_{t_1,\dots,t_d\in \F_p}e_p(t_1a_1+\dots+t_da_d)
\]
we see that,
\begin{align*}
   \mathbbm{1}_{a_1=0,\dots,a_d=0} &= \frac{1}{p-1}\sum_{t\in\F_p, t\neq 0}\mathbbm{1}_{ta_1=0,\dots,ta_d=0} \\
   &=  p^{-d}\sum_{t_1,\dots,t_d\in \F_p}\frac{1}{p-1}\sum_{t\in\F_p, t\neq 0} e_p(t(t_1a_1+\dots+t_da_d))
\end{align*}
Because $X\cdot \mathbf{1} = 0$, we have
\begin{align*}
\P(X\in H) = \P(\wedge_{i=1}^{d-1} X\cdot \Bv_i=0) &= \frac{1}{p^{d-1}}\sum_{t_1,\dots,t_{d-1}\in\F_p}\frac{1}{p-1}\sum_{t\in\F_p,t\neq 0}e_p(t[t_1X\cdot\Bv_1+\dots+t_{d-1}X\cdot\Bv_d]) \\
&=\frac{1}{p^{d-1}}+\frac{1}{p^{d-1}}\sum_{t_i\in\F_p,\text{ not all zero}}\E\Big[\frac{1}{p-1}\sum_{t\in\F_p,t\neq 0}e_p(t(X\cdot \sum_{i}t_i\Bv_i))\Big].
\end{align*}
Observe,
\begin{align*}
\E\Big|\frac{1}{p}\sum_{t\in\F_p, t\neq 0}e_p(t(X\cdot\sum_{i}t_i\Bv_i))\Big| = \Big|\P[X\cdot \sum_{i}t_i\Bv_i = 0]-\frac{1}{p}\Big| \leq \rho_L^j(\sum_{i}t_i\Bv_i).
\end{align*}
By our assumption on $\rho_L^j$,
\[
\E\Big|\frac{1}{p}\sum_{t\in\F_p, t\neq 0}e_p(t(X\cdot\sum_{i}t_i\Bv_i))\Big|\leq \delta.
\]
So, we obtain
\[
\Big|\P(X\in H)-\frac{1}{p^{d-1}}\Big|\leq \frac{p}{p-1}\delta.
\]
\end{proof}

\begin{proof}(of Proposition~\ref{prop:rank})
We have seen from Proposition \ref{prop:non-sparse} that there exists an event $\CE$ with $\P(\CE)\geq 1-O(\exp(-Cn))$ where for all vectors $\Bv\in W_{(1+\alpha)n-k}^\perp$, we have for any $a\in\F_p$,
\begin{align}\label{sparse_constant}
|\supp\left(\Bv-a\cdot\mathbbm{1}\right) \cap [n]| \ge \delta' n.
\end{align}

Proposition \ref{laplacian_concentration} states that $\Bv$ which satisfy equation~\ref{sparse_constant} for all $a$ have,
\[
\rho^j_L(\Bv)\leq \exp(-\delta'n/2p^2), \text{ for all $j$}.
\]
Conditioning on $\CE$, and applying Lemma~\ref{smallrhogood} to $W_{(1+\alpha)n-k}$ and $Y_{(1+\alpha)n-k + 1}$ with $\delta = \exp(-\delta'n/2p^2)$ we obtain,
\begin{align*}
\Big|\P\big(Y_{(1+\alpha)n-k + 1} \in W_{(1+\alpha)n-k} | \CE \wedge W_{(1+\alpha)n-k} \mbox{ has codimension $l$}\big)  - \frac{1}{p^{l-1}}\Big| &\leq \frac{p}{p-1}\exp(-\delta' n/p^2)\\ &= O(\exp(-\delta' n/p^2)).
\end{align*}
Note $\frac{p}{p-1}\leq 2$ so the big-O does not depend on $p$. 
\end{proof}
\section{proof of Theorem \ref{theorem:main}}\label{finalproof}
\begin{proof}
It is known \cite{FG} that for a uniform $(n+1)\times n$ matrix $M_{n}'$ with edge probability $q$, the final corank distribution is given by
\[
\P(\corank(M_n') = 1+k) = \frac{1}{p^{k^2+k}}\frac{\prod_{i=k+2}^{\infty}\left(1-1/p^i\right)}{\prod_{i=1}^k\left(1-1/p^i\right)} + O(q^{-cn}).\
\]
Let $r = \delta n/p^2$ for some constant $\delta>0$ to be fixed later. Applying Proposition \ref{prop:rect-full-rank}, we have that the matrix $M_{(1+\alpha)n-r}$ is full rank with probability at least $1-O(\exp(-Cr))$. Similarly $(M_n')_{n+1-r}$ will be full rank with probability $1-O(\exp(-Cr)$.

With $r$ rows remaining the rank of both models evolve similarly with a (high probability) initial rank of $n-r+1$ for the $(n+1)\times n$ model and $n-r$ for the Laplacian model. In the uniform model if the rank of $(M_n')_\ell$ is $k$ then the rank increases to $k+1$ when we add a row with probability $1-1/p^{n-k}$. In the Laplacian model if the rank of $M_\ell$ for $(1+\alpha-\delta)n\leq \ell < n$ is $k$ then the rank increases to $k+1$ with probability $1-1/p^{n-k-1}+O(\exp(-Cn/p^2))$.

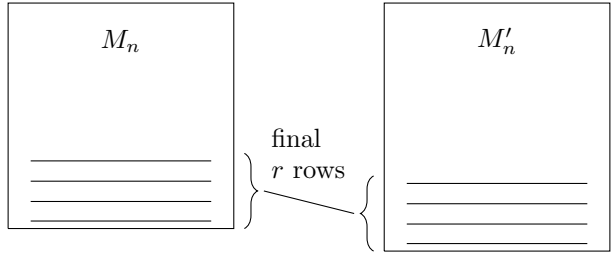
\begin{figure}[h]
	\centering
	\begin{tikzpicture}
        \draw [decorate,decoration={brace,amplitude=5pt},xshift=4pt,yshift=0pt]
        (4,2) -- (4,1) node [black,midway,yshift=0.6cm] {};
        \draw [decorate,decoration={brace,amplitude=5pt},xshift=-4pt,yshift=0pt]
        (6,0.7) -- (6,1.7) node [black,midway,yshift=0.6cm] {};
        \draw (4.4, 1.5) -- (5.6, 1.2);
        \node[text width=1cm] at (5,2) {final\\$r$ rows};

		\draw (1,1) -- (4,1) -- (4,4) -- (1,4) -- (1,1);
		\node at (2.5,3.5) {$M_n$};
		
		\draw (1.3, 1.9) -- (3.7, 1.9);
		\draw (1.3, 1.63) -- (3.7, 1.63);
		\draw (1.3, 1.36) -- (3.7, 1.36);
		\draw (1.3, 1.1) -- (3.7, 1.1);
		
		\draw (6,0.7) -- (9,0.7) -- (9,4) -- (6,4) -- (6,0.7);
		\node at (7.5,3.5) {$M_n'$};
		
		\draw (6.3, 1.6) -- (8.7, 1.6);
		\draw (6.3, 1.33) -- (8.7, 1.33);
		\draw (6.3, 1.06) -- (8.7, 1.06);
		\draw (6.3, 0.8) -- (8.7, 0.8);
	\end{tikzpicture}
	\caption{Laplacian vs. $(n+1)\times n$ iid model}
	\label{fig}
\end{figure}

In other words the two models start in the same place and evolve similarly with high probability. This is illustrated by Figure~\ref{fig}. Summing over the errors for the change of rank over the addition of the final $r$ rows, we achieve a total error
\[
O(\exp(-Cr)) + r\binom{r}{k}O(\exp(-Cn/p^2)).
\]
Now, we take $0<\delta< C$ small enough so that
\[
=O(\exp(-C \delta n/p^2)) + \frac{\delta n}{p^2}\binom{\delta n/p^2}{k}O(\exp(-Cn/p^2)) = O(\exp(-Cn/p^2))
\]
for some constant $C>0$. This yields the final result
\[
\P(\corank (M_{(1+\alpha)n}) = 1+k) = \frac{1}{p^{k^2+k}}\frac{\prod_{i=k+2}^{\infty}\left(1-1/p^i\right)}{\prod_{i=1}^k\left(1-1/p^i\right)} +O(\exp(-Cn/p^2)).
\]
\end{proof}

\section{Conflict of Interest}
On behalf of all authors, the corresponding author states that there is no conflict of interest.
\section{Data Availability Statement}
Data sharing not applicable to this article as no datasets were generated or analysed during the current study.

\end{document}